\RequirePackage{fix-cm}
\documentclass[smallextended]{svjour3}       
\smartqed  
\usepackage{graphicx}
\usepackage{amsmath}
\usepackage{amsfonts}
\usepackage{enumerate}
\usepackage{latexsym}
\begin{document}

\title{How much faster does the best polynomial approximation converge
than Legendre projection?\thanks{This work was supported by National
Natural Science Foundation of China under grant number 11671160.} }

\titlerunning{How much faster does the best approximation converge than Legendre}

\author{Haiyong Wang}


\institute{Haiyong Wang \at
              School of Mathematics and Statistics, Huazhong
University of Science and Technology, Wuhan 430074, P. R. China. \\
              \email{haiyongwang@hust.edu.cn}            \\
             Hubei Key Laboratory of Engineering Modeling and
Scientific Computing, Huazhong University of Science and Technology,
Wuhan 430074, P. R. China.  
            }

\date{Received: date / Accepted: date}

\maketitle

\begin{abstract}
We compare the convergence behavior of best polynomial
approximations and Legendre and Chebyshev projections and derive
optimal rates of convergence of Legendre projections for analytic
and differentiable functions in the maximum norm. For analytic
functions, we show that the best polynomial approximation of degree
$n$ is better than the Legendre projection of the same degree by a
factor of $n^{1/2}$. For differentiable functions such as piecewise
analytic functions and functions of fractional smoothness, however,
we show that the best approximation is better than the Legendre
projection by only some constant factors. Our results provide some
new insights into the approximation power of Legendre projections.
\end{abstract}

\section{Introduction}
\label{intro} The Legendre polynomials are one of the most important
sequences of orthogonal polynomials which have been extensively used
in many branches of scientific computing such as Gauss-type
quadrature, special functions, $p$-version of the finite element
method and spectral methods for differential and integral equations
(see, e.g.,
\cite{canuto2006spectral,davis1984methods,eriksson1986error,gui1986fem,hesthaven2007spectral,osipov2013pswf,shen1994legendre,shen2011spectral,szego1975orthogonal}).
Among these applications, Legendre polynomials are particularly
appealing owing to their superior properties: (i) they have
excellent error properties in the approximation of a globally smooth
function; (ii) quadrature rules based on their zeros or extrema are
optimal in the sense of maximizing the exactness of integrated
polynomials; (iii) they are orthogonal with respect to the uniform
weight function $\omega(x)=1$ which makes them preferable in
Galerkin methods for PDEs.

Let $n\geq0$ be an integer and let $P_n(x)$ denote the Legendre
polynomial of degree $n$ which is normalized by $P_n(1)=1$. The
sequence of Legendre polynomials $\{P_n(x)\}$ forms a system of
polynomials orthogonal over $\mathrm{\Omega}=[-1,1]$ and
\begin{align}\label{def:LegendrePoly}
\int_{-1}^{1} P_n(x) P_m(x) \mathrm{d}x = \frac{2}{2n+1}
\delta_{mn},
\end{align}
where $\delta_{mn}$ is the Kronecker delta
\cite[p.~14]{olver2010nist}. Given a real-valued function $f(x)$
which belongs to a Lipschitz class of order larger than $1/2$ on
$\mathrm{\Omega}$, then it has the following uniformly convergent
Legendre series expansion \cite{suetin1964legendre}
\begin{align}\label{eq:LegExp}
f(x) = \sum_{k=0}^{\infty} a_k P_k(x), \quad  a_k = \left(k +
\frac{1}{2} \right)\int_{-1}^{1} f(x) P_k(x) \mathrm{d}x.
\end{align}
Let $\mathcal{P}_n(f)$ denote the truncated Legendre expansion of
degree $n$, i.e.,
\begin{align}\label{eq:LegExp}
\mathcal{P}_n(f) = \sum_{k=0}^{n} a_k P_k(x).
\end{align}
which is also known as the Legendre projection. It is well known
that this polynomial is the best polynomial approximation to $f(x)$
in the $L^2$ norm with respect to the Legendre weight $\omega(x)=1$.
The computation of the first $n+1$ Legendre coefficients
$\{a_k\}_{k=0}^{n}$ has received much attention over the past decade
and fast algorithms have been developed in
\cite{alpert1991legendre,iserles2011legendre} that require only
$O(n\log n)$ operations and in \cite{townsend2018fast} that require
$O(n\log^2n)$ operations.

Besides Legendre polynomials, another widely used sequence of
orthogonal polynomials is the Chebyshev polynomials, i.e., $T_k(x) =
\cos(k \arccos(x))$. Suppose that $f(x)$ is Dini-Lipschitz
continuous on $\mathrm{\Omega}$, then it has the following uniformly
convergent Chebyshev series \cite[Theorem 5.7]{mason2003chebyshev}
\begin{align}\label{eq:LegExp}
f(x) = \sum_{k=0}^{\infty}{'} c_k T_k(x), \quad  c_k = \frac{2}{\pi}
\int_{-1}^{1} \frac{f(x) T_k(x)}{\sqrt{1-x^2}} \mathrm{d}x,
\end{align}
where the prime indicates that the first term of the sum is halved.
Let $\mathcal{T}_n(f)$ denote the truncated Chebyshev expansion of
degree $n$, i.e.,
\begin{align}\label{eq:LegExp}
\mathcal{T}_n(f) = \sum_{k=0}^{n} a_k T_k(x),
\end{align}
which is also known as the Chebyshev projection. It is well known
that $\mathcal{T}_n(f)$ is the best polynomial approximation to
$f(x)$ in the $L^2$ norm with respect to the Chebyshev weight
$\omega(x)=(1-x^2)^{-1/2}$ and the first $n+1$ Chebyshev
coefficients $\{c_k\}_{k=0}^{n}$ can be evaluated efficiently by
making use of the FFT in only $O(n\log n)$ operations (see, e.g.,
\cite[Section~5.2.2]{mason2003chebyshev}).

Let $\mathcal{B}_n(f)$ denote the best approximation polynomial of
degree $n$ to $f$ on $\mathrm{\Omega}=[-1,1]$ in the maximum norm,
i.e.,
\[
\| f - \mathcal{B}_n(f) \|_{\infty} = \min_{p\in\mathrm{\Pi}_n} \| f
- p \|_{\infty},
\]
where $\mathrm{\Pi}_n$ denotes the space of polynomials of degree at
most $n$. If $f$ is continuous on $\mathrm{\Omega}$, it is well
known that $\mathcal{B}_n(f)$ exists and is unique. From the point
of view of polynomial approximation in the maximum norm, it is clear
that $\mathcal{B}_n(f)$ is more accurate than $\mathcal{P}_n(f)$ and
$\mathcal{T}_n(f)$. However, explicit expressions for
$\mathcal{B}_n(f)$ are generally impossible to obtain since the
dependence of $\mathcal{B}_n(f)$ on $f$ is nonlinear and Remez-type
algorithms, which are realized by iterative procedures, have been
developed for computing $\mathcal{B}_n(f)$ (see, e.g.,
\cite[Chapter~10]{trefethen2013atap}). Although algorithms are
available, they are still time-consuming when $n$ is in the
thousands or higher. Obviously, this leads us to face an inevitable
dilemma of whether the increase in accuracy is sufficient to justify
the extra cost of computing $\mathcal{B}_n(f)$.

With these three approaches, a natural question is: How much better
is the accuracy of $\mathcal{B}_n(f)$ than $\mathcal{T}_n(f)$ and
$\mathcal{P}_n(f)$ in the maximum norm? For the case of
$\mathcal{T}_n(f)$ where $f\in C(\mathrm{\Omega})$, it has been
shown in \cite[Theorem~2.2]{rivlin1981approx} that the maximum error
of $\mathcal{T}_n(f)$ is inferior to that of $\mathcal{B}_n(f)$ by
at most a logarithmic factor, i.e.,
\begin{align}
\|f - \mathcal{T}_n(f) \|_{\infty} \leq \left( \frac{4}{\pi^2}\log n
+ 4 \right) \|f - \mathcal{B}_n(f) \|_{\infty}.
\end{align}
For the case of $\mathcal{P}_n(f)$, it has been widely reported that
the maximum error of $\mathcal{P}_n(f)$ is inferior to that of
$\mathcal{B}_n(f)$ by at
most a factor of $n^{1/2}$ (see, e.g., \cite{qu1988lebesgue,wang2012legendre,wang2018legendre,xiang2012error}). We summarize here existing results 
from two perspectives:
\begin{itemize}
\item For $f\in C(\mathrm{\Omega})$, it is well known that
\begin{align}\label{eq:errorL}
\|f - \mathcal{P}_n(f)\|_{\infty} &\leq \left(1 + \Lambda_n \right)
\|f- \mathcal{B}_n(f) \|_{\infty},
\end{align}
where $\Lambda_n = \sup_{f\neq0} \|\mathcal{P}_n(f)\|_{\infty}/\|f
\|_{\infty}$ is the Lebesgue constant of $\mathcal{P}_n(f)$.
Furthermore, Qu and Wong in \cite[Equation~(1.10)]{qu1988lebesgue}
showed that
\begin{align}
\Lambda_n &= \frac{n+1}{2} \int_{-1}^{1} \left| P_n^{(1,0)}(x)
\right| \mathrm{d}x = \frac{2^{3/2}}{\sqrt{\pi}} n^{1/2} + O(1),
\nonumber
\end{align}
where $P_n^{(1,0)}(x)$ is the Jacobi polynomial of degree $n$ with
$\alpha=1$ and $\beta=0$ and $\alpha,\beta$ are the parameters in
Jacobi polynomials. Hence we can conclude that the rate of
convergence of $\mathcal{P}_n(f)$ is slower than that of
$\mathcal{B}_n(f)$ by a factor of $n^{1/2}$.

\item Under the assumption that $f,f{'},\ldots,f^{(m-1)}$ are absolutely continuous,
$f^{(m)}$ is of bounded variation and $\|f^{(m)}\|_{T}<\infty$ where
$m\geq1$ is an integer and $\|\cdot\|_{T}$ denotes some weighted
semi-norm. It has been shown in
\cite{wang2012legendre,wang2018legendre} that the Legendre
coefficients of $f$ satisfy $|a_k|=O(k^{-m-1/2})$. As a
direct consequence we obtain 
\begin{align}\label{eq:sum}
\|f-\mathcal{P}_n(f)\|_{\infty} \leq \sum_{k=n+1}^{\infty} |a_k| =
O(n^{-m+1/2}),
\end{align}
where we have used the inequality $|P_k(x)|\leq1$ (see, e.g.,
\cite[p.~94]{shen2011spectral}). Notice that the rate of convergence
of $\mathcal{B}_n(f)$ for such functions is $O(n^{-m})$ as
$n\rightarrow\infty$ \cite[Chapter~7]{timan1963approximation}.
Again, we see that the rate of convergence of $\mathcal{P}_n(f)$ is
slower than that of $\mathcal{B}_n(f)$ by a factor of $n^{1/2}$.
\end{itemize}
Is the rate of convergence of $\mathcal{P}_n(f)$ really slower than
$\mathcal{B}_n(f)$ by a factor of $n^{1/2}$? Let us consider a
motivating example $f(x)=|x|$, which is absolutely continuous on
$\mathrm{\Omega}$ and its first-order derivative is of bounded
variation. Moreover, it has been shown in
\cite[Equation~(2.11)]{wang2018legendre} that the Legendre
coefficients of $f$ satisfy the following sharp bound
\begin{align}
|a_{k}| \leq \frac{4}{\sqrt{\pi(2k-3)}} \left(k - \frac{1}{2}
\right)^{-1} = O(k^{-3/2}),
\end{align}
where $k\geq2$ is even and $a_k=0$ when $k$ is odd. We now consider
the rate of convergence of $\mathcal{B}_n(f)$, $\mathcal{T}_n(f)$
and $\mathcal{P}_n(f)$. For $\mathcal{B}_n(f)$ and
$\mathcal{T}_n(f)$, it is well know that their rates of convergence
are $O(n^{-1})$ as $n\rightarrow\infty$ (see, e.g.,
\cite[Chapter~7]{trefethen2013atap}). For $\mathcal{P}_n(f)$,
however, from \eqref{eq:errorL} 
and \eqref{eq:sum} we can deduce that the predicted rate of
convergence of $\mathcal{P}_n(f)$ is only $O(n^{-1/2})$.
Unexpectedly, we observed in \cite[Figure~3]{wang2018legendre} that
the rate of convergence of $\mathcal{P}_n(f)$ is actually
$O(n^{-1})$ as $n\rightarrow\infty$, which is the same as that of
$\mathcal{B}_n(f)$ and $\mathcal{T}_n(f)$. This unexpected
observation suggests that existing results on the rate of
convergence of $\mathcal{P}_n(f)$ may be suboptimal.

In this paper, we aim to investigate the optimal rate of convergence
of $\mathcal{P}_n(f)$ in the maximum norm. For analytic functions,
we show that the optimal rate of convergence of $\mathcal{P}_n(f)$
is indeed slower than that of $\mathcal{B}_n(f)$ and
$\mathcal{T}_n(f)$ by a factor of $n^{1/2}$, although all three
approaches converge exponentially fast. For differentiable functions
such as piecewise analytic functions and functions of fractional
smoothness, however, we shall improve existing results in
\eqref{eq:errorL} and \eqref{eq:sum} and show that the optimal rate
of convergence of $\mathcal{P}_n(f)$ is actually the same as that of
$\mathcal{B}_n(f)$ and $\mathcal{T}_n(f)$, i.e., the accuracy of
$\mathcal{P}_n(f)$ is inferior to that of $\mathcal{B}_n(f)$ by only
some constant factors. This result appears to be new and of
interest.

The rest of this paper is organized as follows. In the next section,
we present some experimental observations on the maximum error of
$\mathcal{P}_n(f)$ with $\mathcal{B}_n(f)$ and $ \mathcal{T}_n(f)$.
In section \ref{sec:OptimalAnalytic}, we analyze the convergence
behavior of $\mathcal{P}_n(f)$ for analytic functions. An explicit
error bound for $\mathcal{P}_n(f)$ is established and it is optimal
in the sense that it can not be improved with respect to $n$. In
section \ref{sec:Piecewise} we analyze the convergence behavior of
$\mathcal{P}_n(f)$ for piecewise analytic functions and functions
with derivatives of bounded variation. We extend our discussion to
functions of fractional smoothness in section \ref{sec:extension}
and give some concluding remarks in section \ref{sec:conclusion}.

\section{Experimental observations}\label{sec:Experiment}
In this section, we present some experimental observations on the
comparison of the rate of convergence of $\mathcal{T}_n(f)$,
$\mathcal{P}_n(f)$ and $\mathcal{B}_n(f)$. In order to quantify more
precisely the difference in the rate of convergence, we define the
ratio of the maximum errors of $\mathcal{B}_n(f)$ to
$\mathcal{P}_n(f)$ and $\mathcal{T}_n(f)$ as
\begin{align}
\mathcal{R}_n^P = \frac{\|f - \mathcal{B}_n(f) \|_{\infty}}{\|f -
\mathcal{P}_n(f) \|_{\infty}}, \qquad  \mathcal{R}_n^T = \frac{\|f -
\mathcal{B}_n(f) \|_{\infty}}{\|f - \mathcal{T}_n(f) \|_{\infty}}.
\end{align}
In our computations, the maximum error of $\mathcal{B}_n(f)$ is
calculated using the Remez algorithm in Chebfun
\cite{driscoll2014chebfun} and the maximum errors of
$\mathcal{P}_n(f)$ and $\mathcal{T}_n(f)$ are calculated by using a
finer grid in $\mathrm{\Omega}=[-1,1]$.

In Figure \ref{fig:ExamI} we show the maximum error of three
approximations as a function of $n$ for the three analytic functions
$f(x)=\exp(x^5),\ln(1.2+x),(1+4x^2)^{-1}$ and $\mathcal{R}_n^P$
scaled by $n^{1/2}$ and $\mathcal{R}_n^T$. From the top row of
Figure \ref{fig:ExamI}, we see that the rate of convergence of
$\mathcal{B}_n(f)$ is almost indistinguishable with that of
$\mathcal{T}_n(f)$. Moreover, both rates of convergence of
$\mathcal{B}_n(f)$ and $\mathcal{T}_n(f)$ are better than that of
$\mathcal{P}_n(f)$. From the bottom row of Figure \ref{fig:ExamI},
we see that each ratio $\mathcal{R}_n^P$ scaled by $n^{1/2}$
approaches a finite asymptote as $n$ grows, which implies that the
rate of convergence of $\mathcal{B}_n(f)$ is faster than that of
$\mathcal{P}_n(f)$ by a factor of $n^{1/2}$. On the other hand, each
ratio $\mathcal{R}_n^T$ approaches a finite asymptote as $n$ grows
($0.6 \leq \mathcal{R}_n^T \leq 0.7$), which implies that
$\mathcal{B}_n(f)$ is better than $\mathcal{T}_n(f)$ by only some
constant factors.

\begin{figure}[ht]
\centering
\includegraphics[trim = 36mm 0mm 0mm 0mm,width=13cm,height=9cm]{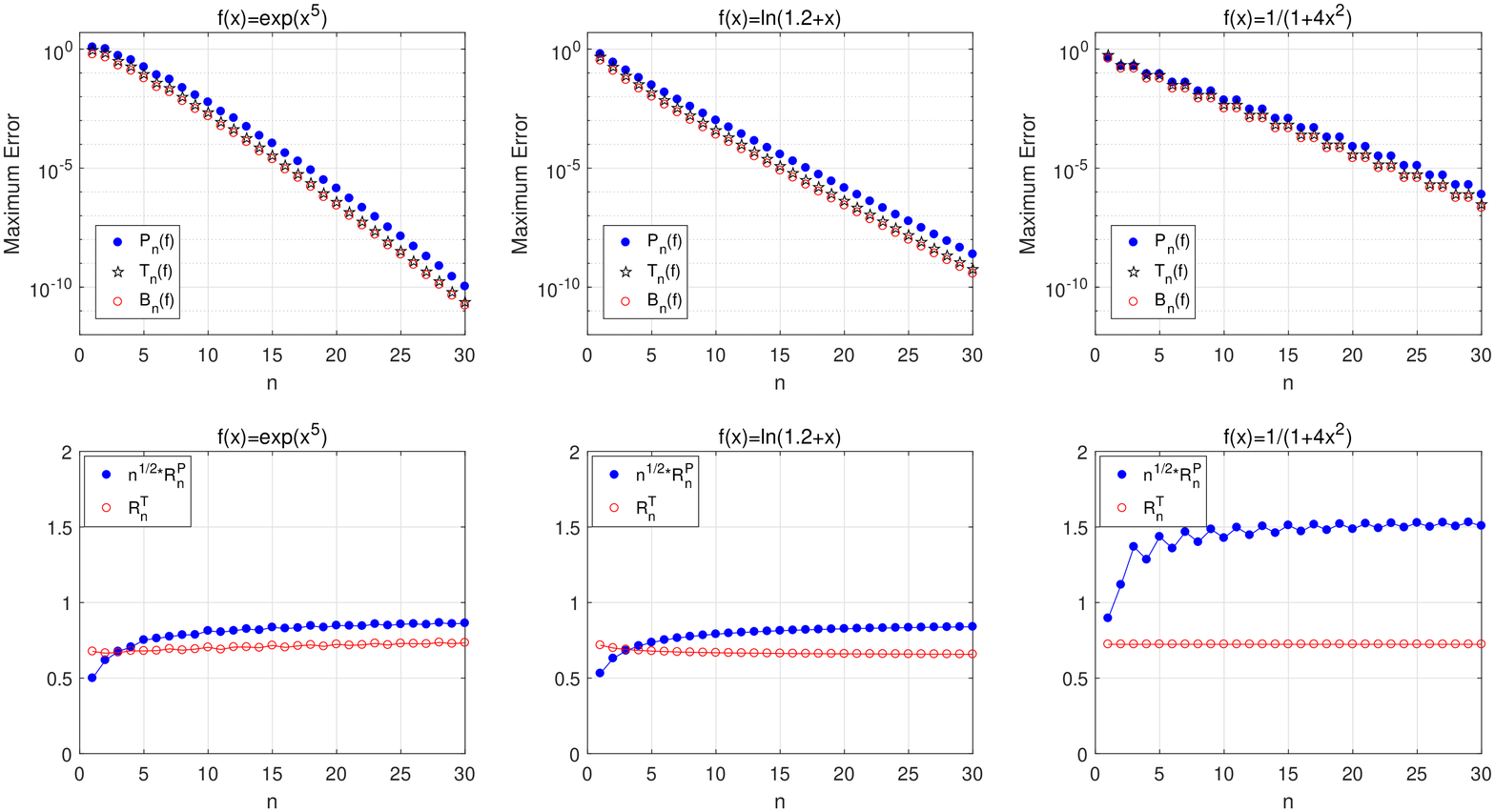}
\caption{Top row shows the log plot of the maximum error of
$\mathcal{B}_n(f)$, $\mathcal{T}_n(f)$ and $\mathcal{P}_n(f)$ for
$f(x)=\exp(x^5)$ (left), $f(x)=\ln(1.2+x)$ (middle) and
$f(x)=1/(1+4x^2)$ (right). Bottom row shows $n^{1/2}
\mathcal{R}_n^P$ and $\mathcal{R}_n^T$. Here $n$ ranges from 1 to
30. } \label{fig:ExamI}
\end{figure}

In Figure \ref{fig:ExamII} we show the maximum error of three
approximations as a function of $n$ for the three differentiable
functions $f(x)=\exp(-1/x^2),(x-\frac{1}{2})_{+}^3,|\sin(5x)|$ and
the corresponding ratios $\mathcal{R}_n^P$ and $\mathcal{R}_n^T$.
For the first test function, it is infinitely differentiable on
$\mathrm{\Omega}$. For the second test function, it is a spline
function whose definition is given in \eqref{eq:spline}. Moreover,
$f\in C^2(\mathrm{\Omega})$ and $f{'''}$ is of bounded variation on
$\mathrm{\Omega}$. For the last function, it is absolutely
continuous and $f{'}$ is of bounded variation on $\mathrm{\Omega}$.
From the top row of Figure \ref{fig:ExamII} we observe that all
three methods $\mathcal{B}_n(f)$, $\mathcal{T}_n(f)$ and
$\mathcal{P}_n(f)$ converge at the same rate. From the bottom row of
Figure \ref{fig:ExamII} we see that each ratio $\mathcal{R}_n^P$ and
$\mathcal{R}_n^T$ oscillates around or converges to a finite
asymptote as $n\rightarrow\infty$, which implies that
$\mathcal{B}_n(f)$ is better than $\mathcal{T}_n(f)$ and
$\mathcal{P}_n(f)$ by only some constant factors (for the last two
functions, note that $\mathcal{R}_n^P$ and $\mathcal{R}_n^T$
approach about $1/2$ as $n\rightarrow\infty$, and thus
$\mathcal{B}_n(f)$ is better than $\mathcal{T}_n(f)$ and
$\mathcal{P}_n(f)$ by a factor of 2).

\begin{figure}[ht]
\centering
\includegraphics[trim = 36mm 0mm 0mm 0mm,width=13cm,height=9cm]{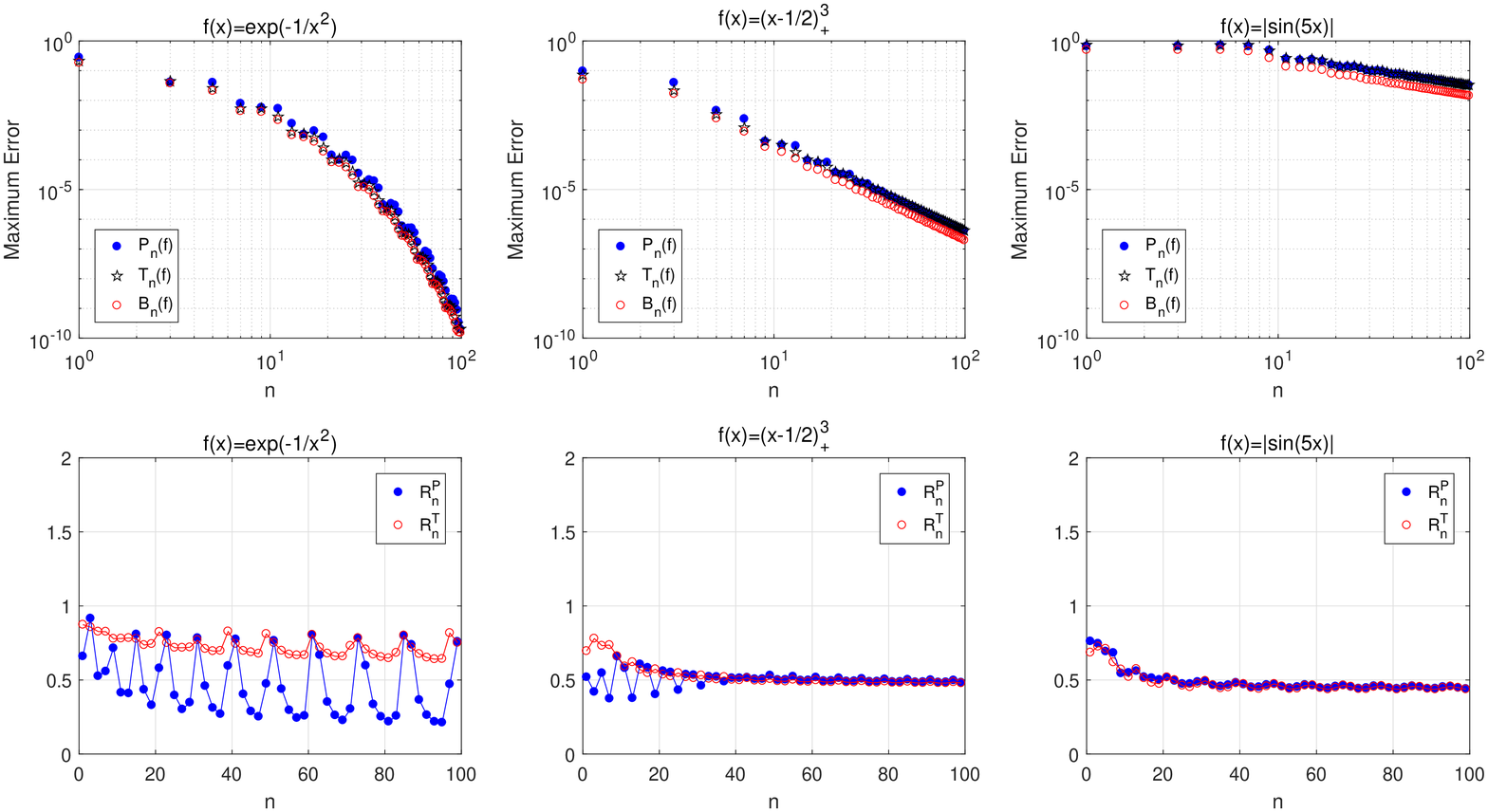}
\caption{Top row shows the log-log plot of the maximum error of
$\mathcal{B}_n(f)$, $\mathcal{T}_n(f)$ and $\mathcal{P}_n(f)$ for
$f(x)=\exp(-1/x^2)$ (left), $f(x)=(x-\frac{1}{2})_{+}^3$ (middle)
and $f(x)=|\sin(5x)|$ (right). Bottom row shows the corresponding
$\mathcal{R}_n^P$ and $\mathcal{R}_n^T$. Here $n$ ranges from 1 to
100. } \label{fig:ExamII}
\end{figure}

In summary, the above observations suggest the following
conclusions:
\begin{itemize}
\item For analytic functions, the rate of convergence of
$\mathcal{B}_n(f)$ is better than that of $\mathcal{T}_n(f)$ by some
constant factors and is better than that of $\mathcal{P}_n(f)$ by a
factor of $n^{1/2}$;

\item For differentiable functions, however, the rate of convergence of
$\mathcal{B}_n(f)$ is better than that of $\mathcal{T}_n(f)$ and
$\mathcal{P}_n(f)$ by only some constant factors.
\end{itemize}
How to explain these observations? Regarding the convergence
behavior of $\mathcal{T}_n(f)$, sharp bounds for its maximum error
have received much attention in recent years. We collect the results
in the following.
\begin{theorem}\label{thm:ChebyRate}
If $f$ is analytic with $|f(z)|\leq M$ in the region bounded by the
ellipse with foci $\pm1$ and major and minor semiaxis lengths
summing to $\rho>1$, then for each $n\geq0$,
\begin{align}\label{eq:ChebAnal}
\|f - \mathcal{T}_n(f)  \|_{\infty} \leq \frac{2M}{\rho^n(\rho-1)}.
\end{align}
If $f,f{'},\ldots,f^{(m-1)}$ are absolutely continuous on
$\mathrm{\Omega}=[-1,1]$ and $f^{(m)}$ is of bounded variation
$V_{m}$ for some integer $m\geq1$, then for each $n\geq m+1$,
\begin{align}\label{eq:ChebDiff}
\|f - \mathcal{T}_n(f)  \|_{\infty} \leq \frac{2V_{m}}{\pi m
(n-m)^{m}}.
\end{align}
\end{theorem}
\begin{proof}
We refer to \cite[Chapter 8]{trefethen2013atap} for the proof of
\eqref{eq:ChebAnal} and \cite[Chapter 7]{trefethen2013atap} for the
proof of \eqref{eq:ChebDiff}.
\end{proof}

A few remarks on Theorem \ref{thm:ChebyRate} are in order.
\begin{remark}
Notice that these functions
$f(x)=\exp(-1/x^2),(x-\frac{1}{2})_{+}^3,|\sin(5x)|$ correspond to
$m=\infty$, $m=3$ and $m=1$, respectively. As a consequence, we can
deduce from \eqref{eq:ChebDiff} that the rates of convergence of
$\mathcal{T}_n(f)$ are $O(n^{-k})$ for any $k\in \mathbb{N}$,
$O(n^{-3})$ and $O(n^{-1})$, respectively. On the other hand, we can
deduce from \cite[Chapter~7]{timan1963approximation} that the rates
of convergence of $\mathcal{B}_n(f)$ for these three functions are
also $O(n^{-k})$ for any $k\in \mathbb{N}$, $O(n^{-3})$ and
$O(n^{-1})$, respectively. Clearly, the rates of convergence of
$\mathcal{T}_n(f)$ and $\mathcal{B}_n(f)$ are of the same order,
which explain the convergence behavior of $\mathcal{T}_n(f)$
observed in Figure \ref{fig:ExamII}. For discussions on the
comparison of $\mathcal{B}_n(f)$ and $\mathcal{T}_n(f)$ when $f$ is
a polynomial of degree larger than $n$, we refer to
\cite{clenshaw1964best}.
\end{remark}

\begin{remark}\label{rk:ChebFrac}
For differentiable functions, the bound \eqref{eq:ChebDiff} is only
optimal for functions with interior singularities of integer-order.
For functions of fractional smoothness, optimal error estimates of
$\mathcal{T}_n(f)$ was recently analyzed in \cite{liu2019optimal} by
introducing fractional Sobolev-type spaces and using the fractional
calculus properties of Gegenbauer functions of fractional degree. We
refer the interested reader to \cite{liu2019optimal} for more
details.
\end{remark}

In the following sections, we shall focus on the convergence
behavior of the Legendre projection $\mathcal{P}_n(f)$ for analytic
and several typical kinds of differentiable functions and present
some theoretical results concerning its optimal rate of convergence.

\section{Optimal rate of convergence of $\mathcal{P}_n(f)$ for analytic functions}
\label{sec:OptimalAnalytic} In this section we study the optimal
rate of convergence of $\mathcal{P}_n(f)$ for analytic functions.
Let $\mathcal{E}_{\rho}$ denote the Bernstein ellipse
\begin{equation}\label{def:Bern}
\mathcal{E}_{\rho} = \left\{ z \in \mathbb{C} ~\bigg|~ z = \frac{u +
u^{-1}}{2},~~ |u| = \rho\geq1 \right\},
\end{equation}
which has the foci at $\pm 1$ and the major and minor semi-axes are
given by $(\rho+\rho^{-1})/2$ and $(\rho-\rho^{-1})/2$,
respectively.

Our starting point is the contour integral expression of the
Legendre coefficients.
\begin{lemma}\label{lem:LegContour}
Suppose that $f$ is analytic in the region bounded by the ellipse
$\mathcal{E}_{\rho}$ for some $\rho>1$, then for each $k\geq0$,
\begin{align}\label{eq:LegContour}
a_k & = \frac{\Gamma(k+1) \Gamma(\frac{1}{2})}{\Gamma(k+\frac{1}{2})
i \pi} \oint_{\mathcal{E}_{\rho}} \frac{f(z)}{ (z \pm \sqrt{z^2 -
1})^{k+1} } {}_2\mathrm{ F}_1\left[\begin{matrix}   k &+ 1,
\frac{1}{2}
\\   k + \frac{3}{2} \hspace{-1cm} &\end{matrix} ; \frac{1}{(z\pm\sqrt{z^2 - 1}
)^{2}} \right]  \mathrm{d}z,
\end{align}
where the sign in $z \pm \sqrt{z^2 - 1}$ is chosen so that
$|z\pm\sqrt{z^2 - 1}|>1$ and $\Gamma(z)$ is the gamma function. Here
${}_2 \mathrm{F}_1(\cdot)$ is the Gauss hypergeometric function
defined by
\[
{}_2\mathrm{ F}_1\left[\begin{matrix} &  a ,  b
\\ c \hspace{-.8cm} &\end{matrix} ; z \right] = \sum_{k=0}^{\infty} \frac{(a)_k
(b)_k}{(c)_k} \frac{z^k}{k!},
\]
where $|z|<1$ and $c\neq0,-1,-2,\ldots$, and $(z)_k$ is the
Pochhammer symbol defined by $(z)_0=1$ and $(z)_{k+1}=(z)_{k}(z+k)$
for $k\geq0$.
\end{lemma}
\begin{proof}
This contour integral was first derived by Iserles in
\cite{iserles2011legendre} for the purpose of designing some fast
algorithms for computing $\{a_k\}_{k=0}^{n}$. The idea of his
derivation is based on writing $a_k$ as a linear combination of
$\{f^{(j)}(0)\}$ and then as an integral transform with a Gauss
hypergeometric function as its kernel. After that, a hypergeometric
transformation was used to replace the original kernel by a new one
that converges rapidly, which finally leads to
\eqref{eq:LegContour}. More recently, a new and simpler approach for
the derivation of \eqref{eq:LegContour} was proposed in
\cite{wang2016gegenbauer} and the idea is simply to rearrange the
Chebyshev coefficients of the second kind. We refer the interested
reader to \cite{iserles2011legendre,wang2016gegenbauer} for more
details.
\end{proof}

In the following, we state some new upper bounds for the Legendre
coefficients, which are simpler but slightly less sharp than the
result stated in \cite{wang2016gegenbauer}. As will be shown later,
these new bounds allow us to establish a new and explicit error
bound for the Legendre projection $\mathcal{P}_n(f)$.
\begin{lemma}\label{lem:LegBound}
Suppose that $f$ is analytic in the region bounded by the ellipse
$\mathcal{E}_{\rho}$ for some $\rho>1$, then for each $k\geq0$,
\begin{align}\label{eq:LegBound}
|a_0| \leq \frac{D(\rho)}{2}, \qquad  |a_k| &\leq D(\rho)
\frac{k^{1/2}}{\rho^{k}}, \quad k\geq1,
\end{align}
where $D(\rho)$ is defined by
\begin{align}\label{eq:D}
D(\rho) &= \frac{\displaystyle 2 L(\mathcal{E}_{\rho})}{\pi
\sqrt{\rho^2-1}} \max_{z\in\mathcal{E}_{\rho}}|f(z)|.
\end{align}
Here $L(\mathcal{E}_{\rho})$ denotes the length of the circumference
of $\mathcal{E}_{\rho}$.
\end{lemma}
\begin{proof}
From Lemma \ref{lem:LegContour}, we immediately obtain
\begin{align}\label{eq:LegBoundS1}
|a_k| &\leq
\frac{\Gamma(k+1)\Gamma(\frac{1}{2})}{\Gamma(k+\frac{1}{2}) \pi}
{}_2\mathrm{ F}_1\left[\begin{matrix} k & + 1, \frac{1}{2}
\\   k + \frac{3}{2}  \hspace{-1cm} &\end{matrix} ; ~ \frac{1}{ \rho^{2}}
\right] \frac{L(\mathcal{E}_{\rho})}{\rho^{k+1}}
\max_{z\in\mathcal{E}_{\rho}}|f(z)|.
\end{align}
Furthermore, for each $k\geq0$ and $\rho>1$, we have
\begin{align}\label{eq:LegBoundS2}
 {}_2\mathrm{ F}_1\left[\begin{matrix} k & + 1, \frac{1}{2}
\\   k + \frac{3}{2} \hspace{-1cm} & \end{matrix};  ~ \frac{1}{\rho^{2}}
\right] \leq {}_2\mathrm{ F}_1\left[\begin{matrix} k & +
\frac{3}{2}, \frac{1}{2}
\\   k + \frac{3}{2} \hspace{-1cm} & \end{matrix};  ~ \frac{1}{\rho^{2}}
\right] &= {}_1\mathrm{ F}_0\left[\begin{matrix} \frac{1}{2}
\\ \hspace{-1cm} & \end{matrix};  ~ \frac{1}{\rho^{2}}
\right] \nonumber \\
&= \left(1 - \frac{1}{\rho^2} \right)^{-1/2}.
\end{align}
Combining \eqref{eq:LegBoundS1} and \eqref{eq:LegBoundS2}, the bound
for $|a_0|$ follows immediately. We now consider the case $k\geq1$.
To establish an explicit bound for the ratio of gamma functions in
\eqref{eq:LegBoundS1}, we define the following sequence
\[
\psi(k) = \frac{\Gamma(k+1)
\Gamma(\frac{1}{2})}{\Gamma(k+\frac{1}{2})} k^{-1/2}.
\]
It can be easily shown that the sequence $\{\psi(k)\}$ is strictly
decreasing. Hence, we obtain
\begin{align}\label{eq:LegBoundS3}
\psi(k)\leq\psi(1) = 2 ~~ \Rightarrow ~~
 \frac{\Gamma(k+1)\Gamma(\frac{1}{2})}{\Gamma(k+\frac{1}{2})} \leq 2
k^{1/2}.
\end{align}
Combining \eqref{eq:LegBoundS1}, \eqref{eq:LegBoundS2} and
\eqref{eq:LegBoundS3} gives the desired result. This completes the
proof.
\end{proof}

\begin{remark}
Sharp bounds for the Legendre coefficients of analytic functions
were studied in
\cite{wang2012legendre,wang2016gegenbauer,xiang2012error,zhao2013sharp}
with different approaches. The new bound \eqref{eq:LegBound} is
slightly less sharp than the latest result stated in \cite[Corollary
4.5]{wang2016gegenbauer} by a factor of up to $2/\pi^{1/2}$($\approx
1.13$) since we have established a uniform bound for $\psi(k)$ in
\eqref{eq:LegBoundS3}. However, the factor $D(\rho)$ in \eqref{eq:D}
is independent of $k$, which is more convenient when applying
\eqref{eq:LegBound} to refine a simple error bound of
$\mathcal{P}_n(f)$, as will be shown below.
\end{remark}

\begin{remark}
The length of the circumference of $\mathcal{E}_{\rho}$ is given by
$L(\mathcal{E}_{\rho}) = 4E(\varepsilon)/\varepsilon$, where
$\varepsilon = 2/(\rho + \rho^{-1})$ and $E(z)$ is the complete
elliptic integral of the second kind (see, e.g.,
\cite[Equation~(19.9.9)]{olver2010nist}). For various approximation
formulas of $L(\mathcal{E}_{\rho})$, we refer to the survey article
\cite{almkvist1988ellipse} for an extensive discussion. Moreover,
sharp bounds of $L(\mathcal{E}_{\rho})$ are also available (see,
e.g., \cite{Jameson2014ellipse}), i.e.,
\begin{align}\label{ineq:bound ellipse}
L(\mathcal{E}_{\rho}) \leq 2 \left(\rho + \frac{1}{\rho} \right) + 2
\left( \frac{\pi}{2} - 1 \right) \left(\rho - \frac{1}{\rho}
\right), \quad \rho\geq1,
\end{align}
and the above inequality becomes an equality when $\rho=1$ or
$\rho\rightarrow\infty$.
\end{remark}

With the above Lemma at hand, we are now able to establish an
explicit error bound for the Legendre projection $\mathcal{P}_n(f)$
in the $L^{\infty}$ norm. Moreover, we show that the derived error
bound is optimal up to a constant factor.
\begin{theorem}\label{thm:OptimalRateAnal}
Suppose that $f$ is analytic in the region bounded by the ellipse
$\mathcal{E}_{\rho}$ for some $\rho>1$. Then, for each $n\geq0$,
\begin{align}\label{eq:BoundAnal}
\|f - \mathcal{P}_n(f) \|_{\infty} \leq \frac{D(\rho)}{\rho^n}
\left[ \frac{(n+1)^{1/2}}{\rho-1} + \frac{(n+1)^{-1/2}}{(\rho-1)^2}
\right].
\end{align}
Up to a constant factor, the bound on the right hand side is optimal
in the sense that it can not be improved in any negative powers of
$n$ further.
\end{theorem}
\begin{proof}
As a consequence of Lemma \ref{lem:LegBound}, we obtain that
\begin{align}\label{eq:LegProjBoundS1}
\|f - \mathcal{P}_n(f) \|_{\infty} \leq \sum_{k=n+1}^{\infty} |a_k|
\leq D(\rho) \sum_{k=n+1}^{\infty} \frac{k^{1/2}}{\rho^{k}}.
\end{align}
For the last sum in \eqref{eq:LegProjBoundS1}, we have
\begin{align}
\sum_{k=n+1}^{\infty} \frac{k^{1/2}}{\rho^{k}} &\leq (n+1)^{-1/2}
\sum_{k=n+1}^{\infty} \frac{k}{\rho^k} = \frac{1}{\rho^n} \left[
\frac{(n+1)^{1/2}}{\rho-1} + \frac{(n+1)^{-1/2}}{(\rho-1)^2}
\right]. \nonumber
\end{align}
This proves the bound \eqref{eq:BoundAnal}.

We now turn to prove the optimality of the bound
\eqref{eq:BoundAnal}. By contradiction suppose that it can be
further improved in a negative power of $n$, i.e.,
\begin{align}\label{eq:LegProjBoundS3}
\|f - \mathcal{P}_n(f) \|_{\infty} \leq  n^{-\gamma}
\frac{D(\rho)}{\rho^n} \left[ \frac{(n+1)^{1/2}}{\rho-1} +
\frac{(n+1)^{-1/2}}{(\rho-1)^2} \right],
\end{align}
where $\gamma>0$. Let us consider a concrete function, e.g., $f(x) =
(x-2)^{-1}$. It is easily seen that this function has a simple pole
at $x=2$ and therefore $\rho\leq 2+\sqrt{3}-\epsilon$, where
$\epsilon>0$ may be taken arbitrary small. On the other hand, using
Lemma \ref{lem:LegContour} and the residue theorem, we can write the
Legendre coefficients of $f(x)$ as
\begin{align}\label{eq:LegProjBoundS4}
a_k & = \frac{\Gamma(k+1)
\Gamma(\frac{1}{2})}{\Gamma(k+\frac{1}{2})} {}_2\mathrm{
F}_1\left[\begin{matrix} k & + 1, \frac{1}{2}
\\   k + \frac{3}{2}  \hspace{-1cm} & \end{matrix} ;  \frac{1}{(2+\sqrt{3})^{2}} \right]
\frac{(-2)}{(2+\sqrt{3})^{k+1}}.
\end{align}
Clearly, $a_k<0$ for all $k\geq0$, and it is easy to check that the
sequence $\{-a_{k}\}_{k=0}^{\infty}$ is strictly decreasing. Now, we
consider the error of the Legendre projection at the point $x=1$. In
view of $P_k(1)=1$ for $k\geq0$, we obtain that
\begin{align}
\left| f(x) - \mathcal{P}_n(f) \right|_{x=1} &=
\sum_{k=n+1}^{\infty} (-a_{k}) \geq -a_{n+1}.  \nonumber
\end{align}
Thus, combining the above bound with \eqref{eq:LegProjBoundS3}
yields
\begin{align}
-a_{n+1} \leq \|f(x) - \mathcal{P}_n(f)\|_{\infty} \leq n^{-\gamma}
\frac{D(\rho)}{\rho^n} \left[ \frac{(n+1)^{1/2}}{\rho-1} +
\frac{(n+1)^{-1/2}}{(\rho-1)^2} \right].
\end{align}
Furthermore, from \eqref{eq:LegProjBoundS4} we can deduce that the
lower bound of $\|f(x) - \mathcal{P}_n(f)\|_{\infty}$ behaves like
$|a_{n+1}|=O( n^{1/2}(2+\sqrt{3})^{-n})$ and the upper bound of
$\|f(x) - \mathcal{P}_n(f)\|_{\infty}$ behaves like
$O(n^{1/2-\gamma}(2+\sqrt{3}-\epsilon)^{-n})$ as
$n\rightarrow\infty$. Clearly, this leads to an obvious
contradiction since the upper bound may be smaller than the lower
bound when $\epsilon$ is sufficiently small. Therefore, we can
conclude that the derived bound \eqref{eq:BoundAnal} is optimal in
the sense that it can not be improved in any negative powers of $n$
further. This completes the proof.
\end{proof}

\begin{remark}
From \cite[p.~131]{cheney1998approximation} we know that
\begin{align}
\frac{\pi}{4} \max_{k\geq n} \left\{ |c_k| \right\} \leq \|f -
\mathcal{B}_n(f) \|_{\infty} \leq \sum_{k=n+1}^{\infty} |c_k|.
\end{align}
Moreover, from \cite[p.~95]{bernstein1912cheb} we know that $|c_k|
\leq 2\max_{z\in\mathcal{E}_{\rho}}|f(z)| \rho^{-k}$, and thus the
rate of convergence of $\mathcal{B}_n(f)$ is $O(\rho^{-n})$ as
$n\rightarrow\infty$, i.e., $\|f - \mathcal{B}_n(f)
\|_{\infty}=O(\rho^{-n})$. Comparing this with \eqref{eq:BoundAnal},
it is easy to see that the rate of convergence of $\mathcal{B}_n(f)$
is $O(n^{1/2})$ faster than that of $\mathcal{P}_n(f)$. Moreover,
comparing \eqref{eq:BoundAnal} and \eqref{eq:ChebAnal}, we see that
the rate of convergence of $\mathcal{T}_n(f)$ is also $O(n^{1/2})$
faster than that of $\mathcal{P}_n(f)$. These explain the
convergence behavior of $\mathcal{P}_n(f),\mathcal{T}_n(f)$ and
$\mathcal{B}_n(f)$ illustrated in Figure \ref{fig:ExamI}.
\end{remark}

\section{Optimal rate of convergence of $\mathcal{P}_n(f)$ for
functions with derivatives of bounded
variation}\label{sec:Piecewise} In this section we study optimal
rate of convergence of $\mathcal{P}_n(f)$ for differentiable
functions with derivatives of bounded variation. We start with the
case of piecewise analytic functions and then extend our discussion
to the case of functions whose $m$th order derivative is of bounded
variation. Throughout this paper, we denote by $K$ a generic
positive constant independent of $n$.

\subsection{Piecewise analytic functions}
We first introduce the definition of piecewise analytic function
(see, e.g., \cite{saff1989poly}).

\begin{definition}\label{def:PiecewiseAnal}
Let $f$ be a piecewise analytic function, by which we mean there
exist a set of points
\[
-1< \xi_1 < \xi_2 < \cdots < \xi_{\ell} < 1, \quad \ell\geq1,
\]
such that the restriction of $f$ to each
$[-1,\xi_1]$,$[\xi_1,\xi_2]$,$\ldots$,$[\xi_{\ell},1]$ has an
analytic continuation to a neighborhood of this closed interval, but
$f$ itself is not analytic at each point $\xi_1,\ldots,\xi_{\ell}$.
In the following discussion, we will denote by
$\mathrm{PA}(\mathrm{\Omega},\vec{\xi})$, where
$\vec{\xi}=(\xi_1,\ldots,\xi_{\ell})^T\in \mathbb{R}^{\ell}$ and
$\cdot^T$ denotes the transpose, the set of piecewise analytic
functions for notational simplicity.
\end{definition}

In order to analyze the convergence behavior of $\mathcal{P}_n(f)$,
we first rewrite it as
\begin{align}\label{eq:LegPn}
\mathcal{P}_n(f) = \sum_{k=0}^{n} P_k(x) \left(k+\frac{1}{2} \right)
\int_{-1}^{1} f(y) P_k(y) \mathrm{d}y = \int_{-1}^{1} f(y) D_n(x,y)
\mathrm{d}y,
\end{align}
where $D_n(x,y)$ is the Dirichlet kernel of Legendre polynomials
defined by
\begin{align}\label{eq:HnExp1}
D_n(x,y) &= \sum_{k=0}^{n} \left(k + \frac{1}{2} \right) P_k(x)
P_k(y).
\end{align}
By means of the Christoffel-Darboux identity for Legendre
polynomials \cite[p.~51]{shen2011spectral}, the Dirichlet kernel can
also be written as
\begin{align}\label{eq:HnExp2}
D_n(x,y) &= \frac{n+1}{2} \left[ \frac{P_{n+1}(x) P_n(y) -
P_{n+1}(y) P_n(x)}{x-y} \right].
\end{align}

In the following we give two useful lemmas.
\begin{lemma}\label{lem:Bound}
For $|x|\leq1$ and $n\geq0$, we have
\begin{align}
|P_n(x)| \leq \sqrt{\frac{2}{\pi}} \left(n + \frac{1}{2}
\right)^{-1/2} \phi_n(x),
\end{align}
where
\begin{align}\label{def:phi}
\phi_n(x) = \min\left\{(1-x^2)^{-1/4}, \sqrt{\frac{\pi}{2}} \left(
n+\frac{1}{2} \right)^{1/2} \right\}.
\end{align}
\end{lemma}
\begin{proof}
Recall the Bernstein-type inequality of Legendre polynomials
\cite{Antonov1981estimate}, i.e.,
\begin{align}
(1-x^2)^{1/4} |P_n(x)| < \sqrt{\frac{2}{\pi}}
\left(n+\frac{1}{2}\right)^{-1/2}, \quad  x\in[-1,1], \nonumber
\end{align}
and the bound is optimal in the sense that the factor
$(n+1/2)^{-1/2}$ can not be improved to $(n+1/2+\epsilon)^{-1/2}$
for any $\epsilon>0$ and the constant $\sqrt{2/\pi}$ is best
possible. On the other hand, recall the well known inequality
$|P_n(x)|\leq1$. Combining these two inequalities give the desired
result.
\end{proof}

\begin{lemma}\label{prop:Hn}
Let $|x|\leq1$ and let $\delta\in(0,1)$.
\begin{enumerate}
\item If $|y|\leq1$, then
\begin{align}\label{eq:HnBound1}
|D_n(x,y)| &\leq \frac{(n+1)^2}{2}.
\end{align}

\item If $|y|\leq 1 - \delta$, then
\begin{align}\label{eq:HnBound2}
|D_n(x,y)| \leq K n,  \quad
n\gg1.
\end{align}
\end{enumerate}
\end{lemma}
\begin{proof}
As for \eqref{eq:HnBound1}, it follows from \eqref{eq:HnExp1} and
the inequality $|P_k(x)|\leq1$. As for \eqref{eq:HnBound2}, we split
our discussion into two cases: $|x-y|<\delta/2$ or
$|x-y|\geq\delta/2$. In the case when $|x-y|<\delta/2$. By
\eqref{eq:HnExp1} and Lemma \ref{lem:Bound} we obtain that
\begin{align}
|D_n(x,y)| \leq \frac{2}{\pi} \sum_{k=0}^{n} \phi_k(x) \phi_k(y)
\leq \frac{2(n+1)}{\pi} (1-x^2)^{-1/4} (1-y^2)^{-1/4}.
\end{align}
For $|y|\leq1-\delta$, it is easily verified that $|x|\leq
1-\delta/2$, and therefore,
\begin{align}\label{eq:HnS1}
|D_n(x,y)| &\leq \frac{2(n+1)}{\pi} \left( 1 - \left(1 -
\frac{\delta}{2} \right)^2 \right)^{-1/4} (1-(1-\delta)^2)^{-1/4}
\nonumber \\
&= \frac{2(n+1)}{\pi} \delta^{-1/2} \left(1 - \frac{\delta}{4}
\right)^{-1/4} (2-\delta)^{-1/4} \nonumber \\
&= O(n).
\end{align}
Next, we consider the case $|x-y|\geq \delta/2$. From
\eqref{eq:HnExp2} and Lemma \ref{lem:Bound} it follows that
\begin{align}\label{eq:HnS2}
|D_n(x,y)| &\leq \frac{n+1}{\delta} \sqrt{\frac{2}{\pi}} \left(
\left(n + \frac{1}{2} \right)^{-1/2} \phi_n(y) + \left(n +
\frac{3}{2} \right)^{-1/2} \phi_{n+1}(y) \right). \nonumber \\
&\leq \frac{2(n+1)}{\delta} \sqrt{\frac{2}{\pi}} \left(n
+ \frac{1}{2} \right)^{-1/2} (1-y^2)^{-1/4} \nonumber \\
&\leq \frac{2}{\delta^{5/4}} \sqrt{\frac{2}{\pi}} (2-\delta)^{-1/4}
(n+1) \left(n + \frac{1}{2} \right)^{-1/2}
\nonumber \\
&= O(n^{1/2}).
\end{align}
Finally, the desired result \eqref{eq:HnBound2} follows from
\eqref{eq:HnS1} and \eqref{eq:HnS2}. This completes the proof.
\end{proof}

We are now ready to state the first main result of this section.
\begin{theorem}\label{thm:LegPieceRate}
Assume that $f\in C^{m-1}(\mathrm{\Omega})\cap
\mathrm{PA}(\mathrm{\Omega},\vec{\xi})$ for some integer $m\in
\mathbb{N}$ and some $\vec{\xi}\in \mathbb{R}^{\ell}$ with
$\ell\geq1$. Then, for $n\gg1$, we have
\begin{align}
\|f - \mathcal{P}_n(f)\|_{\infty} \leq K n^{-m}.
\end{align}
Up to constant factors, the bound on the right hand side is optimal
in the sense that it is the same as that of $\mathcal{B}_n(f)$.
\end{theorem}
\begin{proof}
Since $f\in C^{m-1}(\mathrm{\Omega})$ and is piecewise analytic on
$\mathrm{\Omega}$, we know from \cite[Theorem~3]{saff1989poly} that
there exists a polynomial $p_n$ of degree $n$ such that for all
$x\in\mathrm{\Omega}$
\begin{align}\label{eq:SaffBound}
|f(x) - p_n(x) | \leq \frac{C}{n^{m}} e^{-c n^{\alpha} d(x)^\beta },
\end{align}
where $\alpha\in(0,1)$ and $\beta\geq\alpha$ or $\alpha=1$ and
$\beta>1$, $d(x)=\min_{1\leq k \leq \ell}|x-\xi_k|$ and $C,c$ are
some positive constants. 
Taking $\alpha=\beta\in(0,1)$ and recalling that
$\mathcal{P}_n(f)\equiv f$ whenever $f$ is a polynomial of degree up
to $n$, we immediately obtain
\begin{align}\label{eq:ModFunS1}
|f - \mathcal{P}_n(f)| &\leq |f - p_n| + |\mathcal{P}_n(f-p_n)|
\nonumber \\
&\leq \frac{C}{n^{m}} e^{-c (n d(x))^\alpha} + \frac{C}{n^{m}}
\int_{-1}^{1} e^{-c (nd(y))^{\alpha}} | D_n(x,y) | \mathrm{d}y,
\end{align}
where we have used \eqref{eq:SaffBound} and \eqref{eq:LegPn} in the
last step. It remains to show the last integral in
\eqref{eq:ModFunS1} behaves like $O(1)$ as $n\rightarrow\infty$. For
simplicity of presentation, we denote it by $I$. Moreover, we let
$I_1=[\xi_1-\epsilon,\xi_1+\epsilon],\ldots,I_{\ell}=[\xi_{\ell}-\epsilon,\xi_{\ell}+\epsilon]$,
where $\epsilon>0$ is chosen to be small enough so that these
subintervals $I_1,\ldots,I_{\ell}$ are pairwise disjoint and are
contained in the interior of $\mathrm{\Omega}$, i.e.,
$I_1,\ldots,I_{\ell}\subset\mathrm{\Omega}$. Then
\begin{align}\label{eq:ModFunS2}
I = \sum_{k=1}^{\ell} \int_{I_k} e^{-c (nd(y))^{\alpha}} | D_n(x,y)
| \mathrm{d}y + \int_{\mathrm{\Omega}\backslash\bigcup_{k=1}^{\ell}
I_k} e^{-c (nd(y))^{\alpha}} | D_n(x,y) | \mathrm{d}y.
\end{align}
For the former sum in \eqref{eq:ModFunS2}, notice that
$d(y)=|y-\xi_k|$ when $y\in I_{k}$, and thus we get
\begin{align}
\sum_{k=1}^{\ell} \int_{I_k} e^{-c (nd(y))^{\alpha}} | D_n(x,y) |
\mathrm{d}y &= \sum_{k=1}^{\ell}
\int_{\xi_k-\epsilon}^{\xi_k+\epsilon}
e^{-c (n|y-\xi_k|)^{\alpha}} | D_n(x,y) | \mathrm{d}y   \nonumber \\
&= \sum_{k=1}^{\ell} \int_{-\epsilon}^{\epsilon} e^{-c
(n|t|)^{\alpha}} | D_n(x,t+\xi_k) | \mathrm{d}t,  \nonumber
\end{align}
where we applied the change of variable $y=t+\xi_k$ in the last
step. Furthermore, using \eqref{eq:HnBound2} and a change of
variable $z=nt$, we obtain
\begin{align}\label{eq:ModFunS3}
\sum_{k=1}^{\ell} \int_{I_k} e^{-c (nd(y))^{\alpha}} | D_n(x,y) |
\mathrm{d}y &\leq 2 K \ell n \int_{0}^{\epsilon} e^{-c(nt)^{\alpha}}
\mathrm{d}t \nonumber \\
&\leq 2K\ell \int_{0}^{\infty} e^{-cz^{\alpha}} \mathrm{d}z
\nonumber \\
&= 2K\ell \frac{\Gamma(\alpha^{-1})}{\alpha c^{1/\alpha}}.
\end{align}
For the second term in \eqref{eq:ModFunS2}, notice that $d(y)\geq
\epsilon$ when $y\in\mathrm{\Omega}\backslash\bigcup_{k=1}^{\ell}
I_k$, we obtain
\begin{align}\label{eq:ModFunS4}
\int_{\mathrm{\Omega}\backslash\bigcup_{k=1}^{\ell} I_k} e^{-c
(nd(y))^{\alpha}} | D_n(x,y) | \mathrm{d}y &\leq e^{-c
(n\epsilon)^{\alpha}}
\int_{\mathrm{\Omega}\backslash\bigcup_{k=1}^{\ell} I_k} | D_n(x,y)
|
\mathrm{d}y \nonumber \\
&\leq e^{-c (n\epsilon)^{\alpha}}(n+1)^2,
\end{align}
where we have used \eqref{eq:HnBound1} in the last step. Combining
\eqref{eq:ModFunS1}, \eqref{eq:ModFunS3} and \eqref{eq:ModFunS4}
gives the desired result. This completes the proof.
\end{proof}

\begin{remark}
It is clear that the test functions
$f(x)=(x-\frac{1}{2})_{+}^3,|\sin(5x)|$ are piecewise analytic
functions on $\mathrm{\Omega}$ and they correspond to $m=3$ and
$m=1$, respectively. As a consequence, we can deduce from Theorem
\ref{thm:LegPieceRate} that the rates of convergence of
$\mathcal{P}_n(f)$ are $O(n^{-3})$ and $O(n^{-1})$, respectively.
Clearly, these rates of convergence are the same order as that of
$\mathcal{B}_n(f)$ and $\mathcal{T}_n(f)$, which explain the
convergence behavior of $\mathcal{P}_n(f)$ for these two test
functions observed in Figure \ref{fig:ExamII}.
\end{remark}

\begin{remark}\label{rk:Subint}
In Figure \ref{fig:Pointwise} we plot the pointwise error of
$\mathcal{P}_n(f)$ for the function $f(x)=(x-\frac{1}{2})_{+}$. It
is clear to see that the maximum error of $\mathcal{P}_n(f)$, i.e.,
$\|f-\mathcal{P}_n(f) \|_{\infty}$, is achieved at the singularity
of $f$. Moreover, we also observe that the accuracy of
$\mathcal{P}_n(f)$ is much more accurate than $\mathcal{B}_n(f)$
except at the very small neighborhood of the singularity. A similar
phenomenon for Chebyshev interpolants has been observed in
\cite[Chapter~16]{trefethen2013atap}.
\end{remark}

\begin{figure}[ht]
\centering
\includegraphics[width=6.2cm,height=5.cm]{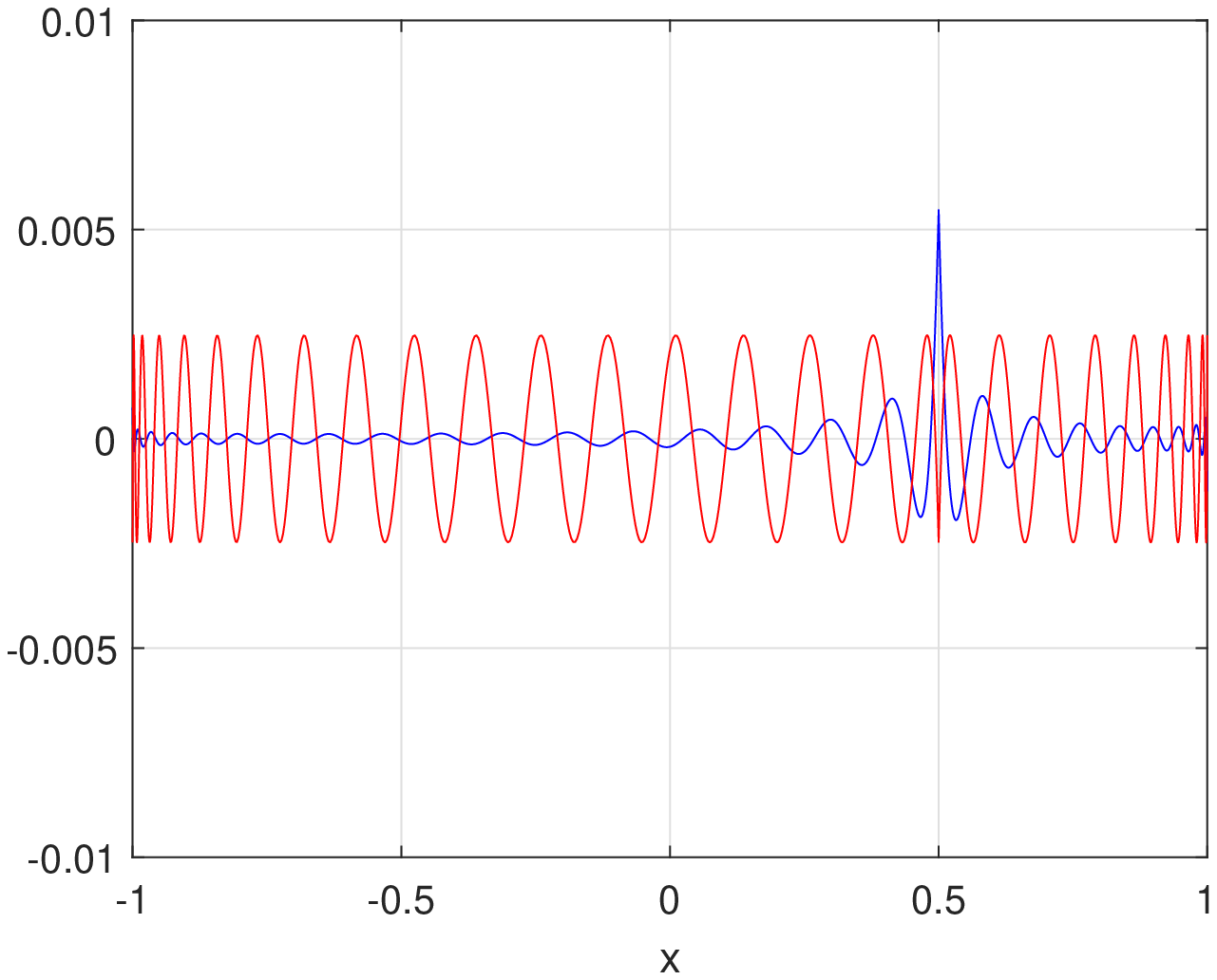}~
\includegraphics[width=6.2cm,height=5.cm]{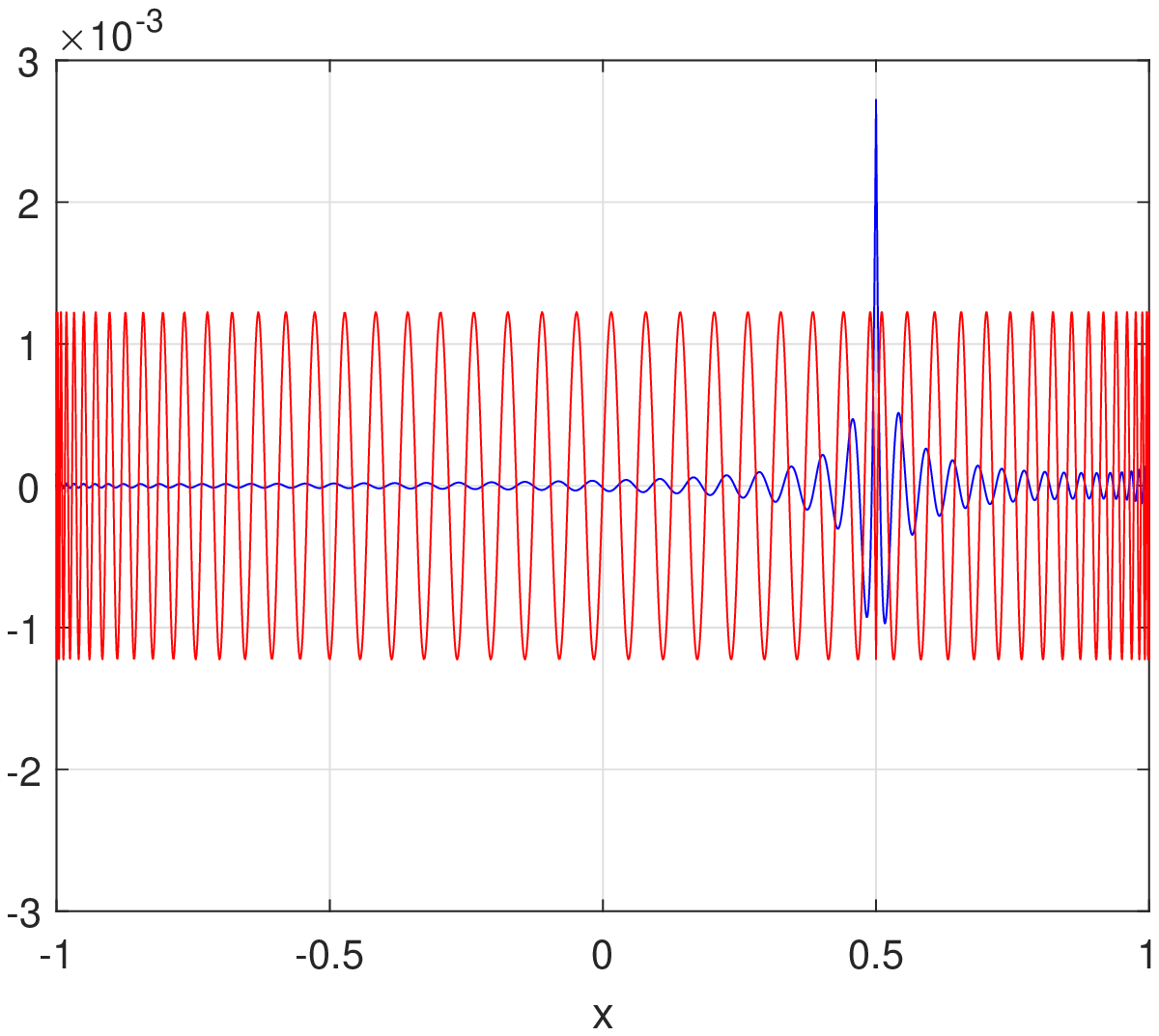}
\caption{Pointwise error of $\mathcal{P}_n(f)$ (blue) and
$\mathcal{B}_n(f)$ (red) for $n=50$ (left) and $n=100$ (right). Here
we choose $f(x)=(x-\frac{1}{2})_{+}$. } \label{fig:Pointwise}
\end{figure}

\subsection{Differentiable functions with derivatives of bounded variation}
In this section we consider the case of differentiable functions
with derivatives of bounded variation. Specifically, let $m\geq1$ be
an integer and introduce the function space
\begin{align}\label{def:FunSpace}
H_m = \left\{f~|~ f,f',\ldots,f^{(m-1)}\in \mathrm{AC(\Omega)}, ~~
f^{(m)}\in \mathrm{BV(\Omega)}  \right\},
\end{align}
where $\mathrm{AC(\Omega)}$ and $\mathrm{BV(\Omega)}$ denote the
space of absolutely continuous functions and the space of bounded
variation functions on $\mathrm{\Omega}$, respectively. This space
is preferable when developing error estimates for various orthogonal
polynomial approximations to differentiable function (see, e.g.,
\cite{liu2019optimal,trefethen2013atap,wang2018legendre,xiang2018jacobi}).
For each $f\in\mathrm{PA}(\mathrm{\Omega},\vec{\xi})\cap
C^{m-1}(\mathrm{\Omega})$, it is easy to see that the restriction of
$f^{(m+1)}$ on each
$[-1,\xi_1]$,$[\xi_1,\xi_2]$,$\ldots$,$[\xi_{\ell},1]$ is continuous
and bounded, and therefore the total variation of $f^{(m)}$ on
$\mathrm{\Omega}$ is finite. Hence, we can deduce that
$\mathrm{PA}(\mathrm{\Omega},\vec{\xi})\cap
C^{m-1}(\mathrm{\Omega})$ is a subset of $H_m$. In the following we
will extend our analysis to the function space $H_m$.

Since $f=\mathcal{P}_n(f)$ for $f\in\mathcal{P}_n$, using the Peano
kernel theorem \cite[Section~4.2]{brass2011quad} we obtain
\begin{align}\label{eq:ErrExp}
f(x) - \mathcal{P}_n(f) = \int_{-1}^{1} f^{(m)}(t) K_{m}(x,t)
\mathrm{d}t,
\end{align}
where $K_{m}(x,t)$ is the Peano kernel defined by
\begin{align}\label{eq:PeanoKernel}
K_{m}(x,t) = \frac{(x-t)_{+}^{m-1} -
\mathcal{P}_n((x-t)_{+}^{m-1})}{(m-1)!},
\end{align}
and
\begin{align}\label{eq:spline}
(x)_{+}^{r} = \left\{
            \begin{array}{ll}
0, & \hbox{$x\leq 0$,}   \\[8pt]
x^r, & \hbox{$x>0$.}
            \end{array}
            \right.
\end{align}

We now state some properties of the Peano kernel.
\begin{lemma}\label{lem:PeanoKernel}
Let $K_m(x,t)$ be the Peano kernel defined in
\eqref{eq:PeanoKernel}. Then for $x\in\mathrm{\Omega}$ and $n\geq
m-1$ we have
\begin{enumerate}
\item[(1)] For $m\geq2$, then $K_m(x,\pm1) = 0$. When $m=1$, then $K_1(x,1) = 0$.

\item[(2)] For each $m\geq2$, then $
\frac{\mathrm{d}}{\mathrm{d}t} K_{m}(x,t) = - K_{m-1}(x,t).$

\item[(3)] For $n\geq m$, we have for any $q\in \mathcal{P}_{n-m}$
that $\int_{-1}^{1} q(t) K_m(x,t) \mathrm{d}t = 0$.

\item[(4)] For $x,t\in[-1,1]$ and $m\geq2$, we have $
\| K_m(x,t) \|_{\infty} \leq K n^{-m+1}$.
\end{enumerate}
\end{lemma}
\begin{proof}
For the first assertion, notice that $(x-1)_{+}^{m-1}=0$ and
$(x+1)_{+}^{m-1}=(x+1)^{m-1}$ when $m\geq2$. Therefore,
$K_m(x,\pm1)=0$. When $m=1$, notice that $(x-1)_{+}^{0}=0$, the
desired result follows. For the second assertion, differentiating
the Peano kernel with respect to $t$ yields
\begin{align}
\frac{\mathrm{d}}{\mathrm{d}t} K_{m}(x,t) = - \frac{(x-t)_{+}^{m-2}
- \mathcal{P}_n((x-t)_{+}^{m-2})}{(m-2)!} = - K_{m-1}(x,t).
\nonumber
\end{align}
This proves the second assertion. For the third assertion, we notice
that $f\equiv\mathcal{P}_n(f)$ whenever $f\in\mathcal{P}_n$. Setting
$f=q\in \mathcal{P}_n$ in \eqref{eq:ErrExp} gives
\begin{align}
q - \mathcal{P}_n(q) &= \int_{-1}^{1} q^{(m)}(t) K_m(x,t)
\mathrm{d}t = 0.  \nonumber
\end{align}
Since $q\in \mathcal{P}_n$ is arbitrary, this proves the third
assertion. For the last assertion, we note that $(x-t)_{+}^{m-1}$ is
a piecewise analytic function and $(x-t)_{+}^{m-1}\in
C^{m-2}[-1,1]$. The desired result follows from Theorem
\ref{thm:LegPieceRate}. This ends the proof.
\end{proof}

We are now ready to state the second main result of this section.
\begin{theorem}\label{thm:LegDiffRate}
Assume that $f\in H_m$ for some integer $m\geq1$. Then, we have
\begin{align}\label{eq:LegDiffRate}
\|f - \mathcal{P}_n(f) \|_{\infty} \leq K n^{-m}.
\end{align}
\end{theorem}
\begin{proof}
Applying the second assertion of Lemma \ref{lem:PeanoKernel} and
integrating by parts, we obtain
\begin{align}
f(x) - \mathcal{P}_n(f) &= - \int_{-1}^{1} f^{(m)}(t)
\frac{\mathrm{d}}{\mathrm{d}t} K_{m+1}(x,t) \mathrm{d}t
\nonumber \\
&= - \left[ f^{(m)}(t) K_{m+1}(x,t)\big|_{-1}^{1} - \int_{-1}^{1}
K_{m+1}(x,t) \mathrm{d} f^{(m)}(t) \right]  \nonumber \\
&= \int_{-1}^{1} K_{m+1}(x,t) \mathrm{d}f^{(m)}(t), \nonumber
\end{align}
where the last integral is understood as a Riemann-Stieltjes
integral and we have used the first assertion of Lemma
\ref{lem:PeanoKernel} in the last step. Furthermore, using the
inequality of Riemann-Stieltjes integral, we arrive at
\begin{align}
\|f(x) - \mathcal{P}_n(f)\|_{\infty} &\leq \|K_{m+1}(x,t)\|_{\infty}
V(f^{(m)}).  \nonumber
\end{align}
where $V(f^{(m)})$ is the total variation of $f^{(m)}$. The desired
result follows from the last assertion of Lemma
\ref{lem:PeanoKernel}.
\end{proof}

\begin{remark}
For the test function $f(x)=\exp(-1/x^2)$, it is infinitely
differentiable on $\mathrm{\Omega}$ and $f^{(m)}\in
\mathrm{BV(\Omega)}$ for every $m\in \mathbb{N}$. Thus, we can
deduce from Theorem \ref{thm:LegDiffRate} that the rate of
convergence of $\mathcal{P}_n(f)$ is $O(n^{-m})$ for any $m\in
\mathbb{N}$. Moreover, for the other two test functions
$f(x)=(x-\frac{1}{2})_{+}^3,|\sin(5x)|$, they can also be viewed as
differentiable functions with derivatives of bounded variation and
they correspond to $m=3$ and $m=1$, respectively. Therefore, we can
deduce from Theorem \ref{thm:LegDiffRate} that the rate of
convergence of $\mathcal{P}_n(f)$ is $O(n^{-3})$ and $O(n^{-1})$,
respectively. Clearly, these results explain why the rate of
convergence of $\mathcal{P}_n(f)$ is the same as that of
$\mathcal{B}_n(f)$ observed in Figure \ref{fig:ExamII}.
\end{remark}

\section{Extension}\label{sec:extension}
In this section we extend our discussion to functions of fractional
smoothness. We shall restrict our attention to some model functions
for the sake of brevity and their results will shed light on the
investigation of more complicated functions.

\subsection{Functions with an interior singularity of fractional order}
Consider the function $f(x)=|x-x_0|^{\alpha}$, where $x_0\in(-1,1)$
and $\alpha>0$ is not an integer. Clearly, this function has an
interior singularity of fractional order. To derive the optimal rate
of convergence of $\mathcal{P}_n(f)$, we shall combine the
asymptotic estimate of the Legendre coefficients of $f$ and the
observation in Remark \ref{rk:Subint}.

Using \cite[Equation~(7.232.3)]{gradshteyn2007table}, we see that
\begin{align}\label{eq:LegInterSing}
a_k &= \left( k + \frac{1}{2} \right) \int_{-1}^{1} |x_0-x|^{\alpha}
P_k(x) \mathrm{d}x \nonumber \\
&= \left( k + \frac{1}{2} \right) \left[ \int_{-1}^{x_0}
(x_0-x)^{\alpha} P_k(x) \mathrm{d}x + \int_{x_0}^{1}
(x-x_0)^{\alpha} P_k(x) \mathrm{d}x \right] \nonumber \\
&= \left( k + \frac{1}{2} \right)
\frac{\Gamma(\alpha+1)\Gamma(k+1)}{\Gamma(k+\alpha+2)} \left[
(1-x_0)^{\alpha+1} P_k^{(\alpha+1,-\alpha-1)}(x_0)  \right.
\nonumber \\
&~~~~~~~~~~~~~~~ \left. + (-1)^k (1+x_0)^{\alpha+1}
P_k^{(\alpha+1,-\alpha-1)}(-x_0) \right],
\end{align}
where $P_k^{(\alpha+1,-\alpha-1)}(x)$ is the Jacobi polynomial of
degree $k$. From \cite[Theorem~8.21.8]{szego1975orthogonal} we know
that $P_k^{(\alpha,\beta)}(x)=O(k^{-1/2})$ where $x\in(-1,1)$ and
$\alpha,\beta$ are arbitrary real numbers. Combining this result
with the asymptotic behavior of the ratio of gamma functions
\cite[Equation~(5.11.12)]{olver2010nist}, we obtain the estimate
$a_k = O(k^{-\alpha-1/2})$. On the other hand, as observed in Remark
\ref{rk:Subint}, the maximum error of $\mathcal{P}_n(f)$ is achieved
in a small neighborhood of the singularity $x=x_0$. Using the
Laplace-Heine formula of the Legendre polynomials
\cite[Theorem~8.21.1]{szego1975orthogonal}, i.e.,
$P_k(x)=O(k^{-1/2})$ where $x\in(-1,1)$, we see at once that
\begin{align}
\|f - \mathcal{P}_n(f) \|_{\infty} \leq \sum_{k=n+1}^{\infty} |a_k|
|P_k(x)| = \sum_{k=n+1}^{\infty} O(k^{-\alpha-1}) = O(n^{-\alpha}).
\end{align}
Moreover, this rate of convergence is optimal in the sense that it
is the same as that of $\mathcal{B}_n(f)$ up to constant factors
(see, e.g., \cite[p.~410]{timan1963approximation}). Regarding
$\mathcal{T}_n(f)$, it has been shown in
\cite[Equation~(4.61)]{liu2019optimal} that the optimal rate of
convergence of $\mathcal{T}_n(f)$ is also $O(n^{-\alpha})$. Thus,
$\mathcal{T}_n(f)$, $\mathcal{B}_n(f)$ and $\mathcal{P}_n(f)$ have
the same rate of convergence for functions with an interior
singularity of fractional order.

\begin{figure}[ht]
\centering
\includegraphics[trim = 36mm 0mm 0mm 0mm,width=13cm,height=9cm]{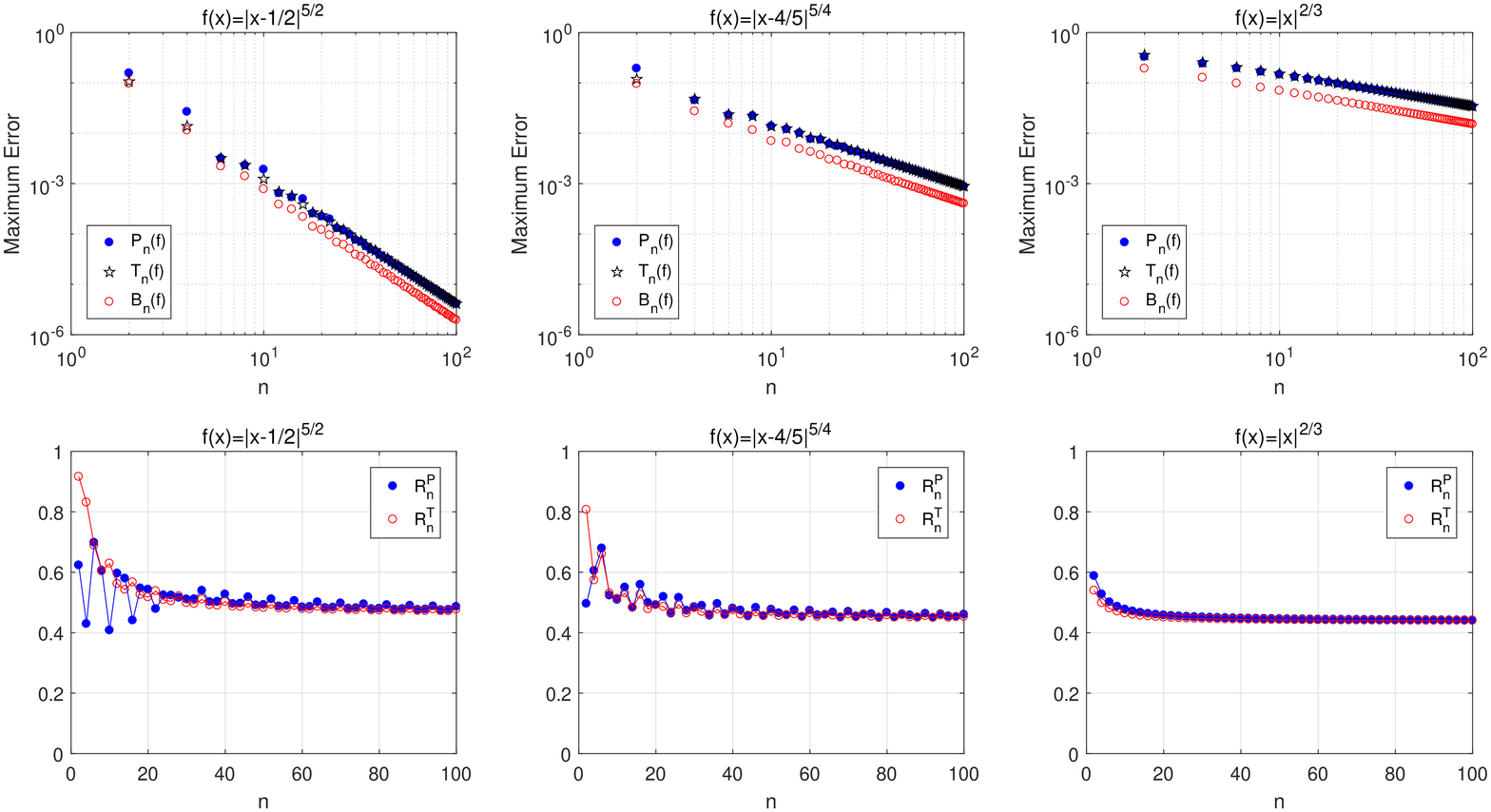}
\caption{Top row shows the log plot of the maximum error of
$\mathcal{B}_n(f)$, $\mathcal{T}_n(f)$ and $\mathcal{P}_n(f)$ for
$f(x)=|x-\frac{1}{2}|^{5/2}$ (left), $f(x)=|x-\frac{4}{5}|^{5/4}$
(middle) and $f(x)=|x|^{2/3}$ (right). Bottom row shows the
corresponding $\mathcal{R}_n^P$ and $\mathcal{R}_n^T$. Here $n$
ranges from 2 to 100. } \label{fig:ExamIV}
\end{figure}

In Figure \ref{fig:ExamIV} we show the maximum error of three
methods as a function of $n$ for the three functions
$f(x)=|x-\frac{1}{2}|^{5/2},|x-\frac{4}{5}|^{5/4},|x|^{2/3}$ and the
corresponding ratios $\mathcal{R}_n^P$ and $\mathcal{R}_n^T$. From
the top row of Figure \ref{fig:ExamIV} we see that all three methods
$\mathcal{B}_n(f)$, $\mathcal{T}_n(f)$ and $\mathcal{P}_n(f)$ indeed
converge at the same rate. Moreover, the accuracy of
$\mathcal{T}_n(f)$ and $\mathcal{P}_n(f)$ is indistinguishable. From
the bottom row of Figure \ref{fig:ExamIV} we see that each ratio
$\mathcal{R}_n^P$ and $\mathcal{R}_n^T$ approaches a constant value
as $n\rightarrow\infty$, which confirms that $\mathcal{B}_n(f)$ is
better than $\mathcal{T}_n(f)$ and $\mathcal{P}_n(f)$ by only some
constant factors (for the three test functions,
$\mathcal{R}_n^P,\mathcal{R}_n^T\in[0.44,0.49]$ as
$n\rightarrow\infty$ and thus $\mathcal{B}_n(f)$ is better than
$\mathcal{T}_n(f)$ and $\mathcal{P}_n(f)$ by a factor of up to 2.3).

\subsection{Functions with endpoint singularities}
Consider the functions $f_{\alpha}(x)=(1\pm x)^{\alpha}$, where
$\alpha>0$ is not an integer. From \cite[Equation
(7.311.3)]{gradshteyn2007table} and setting $\lambda=1/2$, closed
forms of the Legendre coefficients are given by
\begin{align}\label{eq:EndS1}
a_k = (\pm1)^k \frac{2^{\alpha} \Gamma(\alpha+1)^2
(2k+1)}{\Gamma(\alpha+1-k)\Gamma(\alpha+2+k)}, \quad  k\geq0.
\end{align}
Furthermore, combining the reflection
formula \cite[Equation~(5.5.3)]{olver2010nist} and the asymptotic
behavior of the ratio of gamma functions
\cite[Equation~(5.11.12)]{olver2010nist}, we can deduce that
\begin{align}
a_k = (-1)(\mp1)^k \frac{2^{\alpha} \sin(\alpha\pi)
\Gamma(\alpha+1)^2 (2k+1) \Gamma(k-\alpha)}{\pi \Gamma(k+\alpha+2)}
= O(k^{-2\alpha-1}). \nonumber
\end{align}
An important observation is that the sequence $\{a_k\}_{k>\alpha}$
has the same constant sign when $f_{\alpha}(x)=(1-x)^{\alpha}$ and
has alternating signs when $f_{\alpha}(x)=(1+x)^{\alpha}$. Recall
$P_k(\pm1)=(\pm1)^k$, we can deduce that the maximum error of
$\mathcal{P}_n(f_{\alpha})$ is taken at $x=1$ for
$f_{\alpha}(x)=(1-x)^{\alpha}$ and at $x=-1$ for
$f_{\alpha}(x)=(1+x)^{\alpha}$. Therefore, we obtain for $n\geq
\lfloor \alpha \rfloor$ that
\begin{align}\label{eq:EndS2}
\|f_{\alpha} - \mathcal{P}_n(f_{\alpha}) \|_{\infty} =
\sum_{k=n+1}^{\infty} |a_k| = O(n^{-2\alpha}).
\end{align}
We remark that this result is optimal since the rate of convergence
of $\mathcal{B}_n(f_{\alpha})$ is $O(n^{-2\alpha})$ (see, e.g.,
\cite[p.~411]{timan1963approximation}). Moreover, from
\cite{liu2019optimal} we know that the rate of convergence of
$\mathcal{T}_n(f)$ is also $O(n^{-2\alpha})$. Thus, these three
approaches $\mathcal{B}_n(f_{\alpha})$, $\mathcal{P}_n(f_{\alpha})$
and $\mathcal{T}_n(f_{\alpha})$ converge at the same rate.

In Figure \ref{fig:ExamIII} we show the maximum error of
$\mathcal{B}_n(f)$, $\mathcal{T}_n(f)$ and $\mathcal{P}_n(f)$ as a
function of $n$ for the three functions
$f(x)=(1+x)^{5/2},(1-x^2)^{3/2},\cos^{-1}(x)$ and the corresponding
ratios $\mathcal{R}_n^P$ and $\mathcal{R}_n^T$. From the top row of
Figure \ref{fig:ExamIII} we see that all three methods indeed
converge at the same rate. From the bottom row of Figure
\ref{fig:ExamIII} we see that each ratio $\mathcal{R}_n^P$ and
$\mathcal{R}_n^T$ converges to a finite asymptote as
$n\rightarrow\infty$, which means that $\mathcal{B}_n(f)$ is better
than $\mathcal{T}_n(f)$ and $\mathcal{P}_n(f)$ by only some constant
factors (for these three test functions,
$\mathcal{R}_n^P\in[0.17,0.29]$ and $\mathcal{R}_n^T\in[0.44,0.49]$
as $n\rightarrow\infty$ and thus $\mathcal{B}_n(f)$ is better than
$\mathcal{P}_n(f)$ by at most a factor of $5.9$ and is better than
$\mathcal{T}_n(f)$ by at most a factor of 2.3).

\begin{figure}[ht]
\centering
\includegraphics[trim = 36mm 0mm 0mm 0mm,width=13cm,height=9cm]{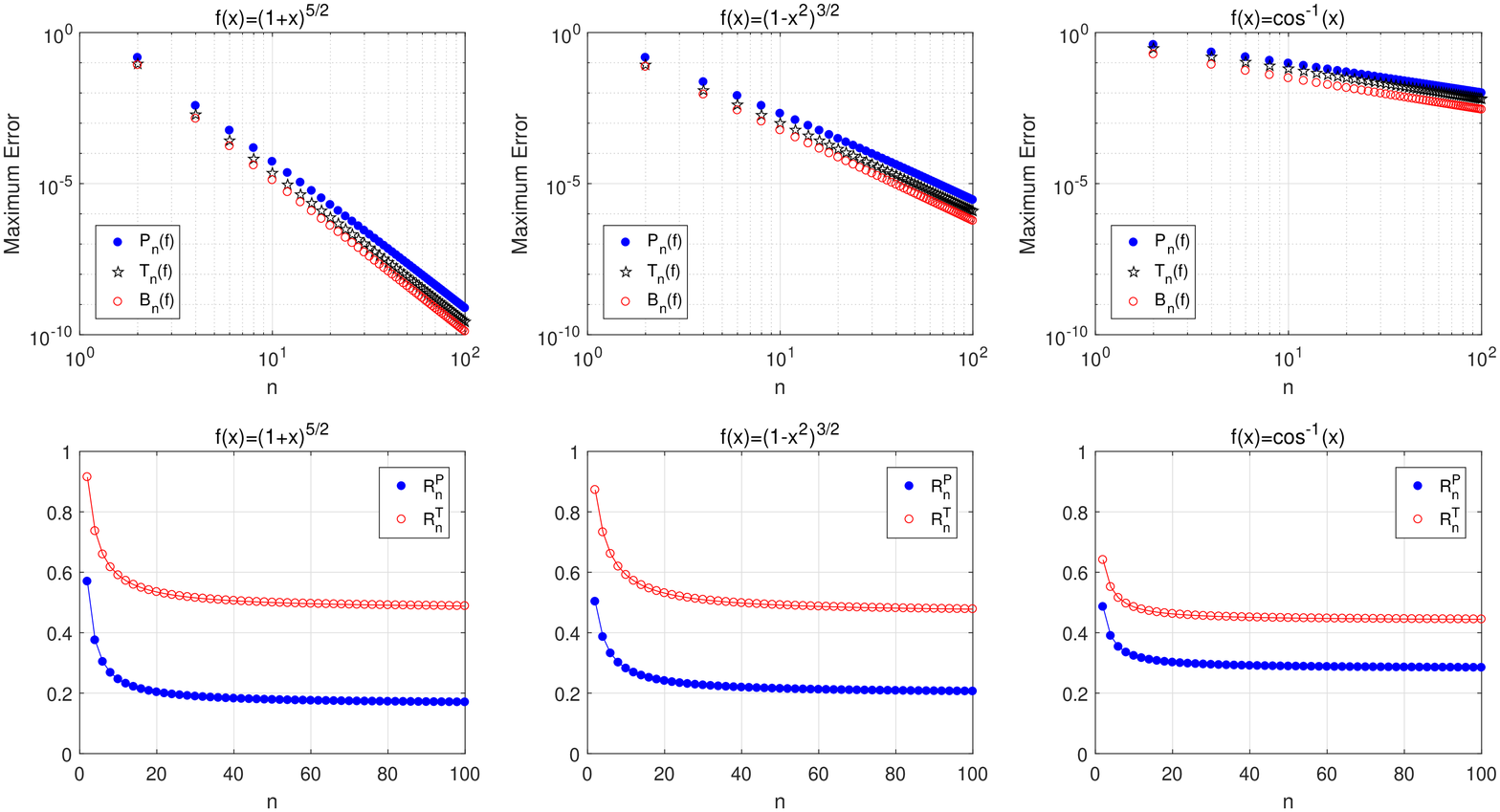}
\caption{Top row shows the log plot of the maximum error of
$\mathcal{B}_n(f)$, $\mathcal{T}_n(f)$ and $\mathcal{P}_n(f)$ for
$f(x)=(1+x)^{5/2}$ (left), $f(x)=(1-x^2)^{3/2}$ (middle) and
$f(x)=\cos^{-1}(x)$ (right). Bottom row shows the corresponding
$\mathcal{R}_n^P$ and $\mathcal{R}_n^T$. Here $n$ ranges from 2 to
100. } \label{fig:ExamIII}
\end{figure}

\begin{remark}
For $f_{\alpha}(x)$, it has been shown in
\cite[Theorem~5.10]{wang2016gegenbauer} that
\begin{align}\label{eq:LegCheb}
\frac{a_k}{c_k} =
\frac{\Gamma(\alpha+1)}{\Gamma(\alpha+\frac{1}{2})} \pi^{1/2} +
O(k^{-1}).
\end{align}
It is easy to verify that the first term on the right hand side is
always greater than one for $\alpha>0$ and is strictly increasing as
$\alpha$ grows. Moreover, similar to the Legendre case, we can show
that the maximum error of $\mathcal{T}_n(f)$ is also achieved at
$x=\mp1$ for $f_{\alpha}(x)=(1\pm x)^{\alpha}$, i.e.,
$\|f_{\alpha}-\mathcal{T}_n(f_{\alpha}) \|_{\infty} =
\sum_{k=n+1}^{\infty}|c_k|$. Combining this with \eqref{eq:EndS2}
and \eqref{eq:LegCheb}, we can deduce that
$\mathcal{T}_n(f_{\alpha})$ is better than
$\mathcal{P}_n(f_{\alpha})$ by a constant factor of
$\frac{\Gamma(\alpha+1)}{\Gamma(\alpha+\frac{1}{2})} \pi^{1/2}$ as
$n\rightarrow\infty$. This means that the larger $\alpha$, the
better the accuracy of $\mathcal{T}_n(f_{\alpha})$ than
$\mathcal{P}_n(f_{\alpha})$, and this phenomenon can be seen clearly
from the bottom row of Figure \ref{fig:ExamIII}.
\end{remark}

\section{Concluding remarks}\label{sec:conclusion}
In this paper we have studied the optimal rate of convergence of
Legendre projections $\mathcal{P}_n(f)$ in the $L^{\infty}$ norm for
analytic and differentiable functions. For analytic functions, we
showed that the optimal rate of convergence of $\mathcal{B}_n(f)$ is
faster than that of $\mathcal{P}_n(f)$ by a factor of $n^{1/2}$. For
differentiable functions such as piecewise analytic functions and
functions of fractional smoothness, however, we improved the
existing results and showed that the rate of convergence of
$\mathcal{B}_n(f)$ is better than that of $\mathcal{P}_n(f)$ by only
some constant factors (the factor is between $2$ to $6$ for most of
examples displayed in this paper). Our results provide new insights
into the approximation power of $\mathcal{P}_n(f)$.

Finally, we present some problems for future research:
\begin{itemize}
\item In Figure \ref{fig:Pointwise}, we have illustrated the pointwise error of
$\mathcal{P}_n(f)$. It can be seen that $\mathcal{P}_n(f)$ converges
actually much faster than $\mathcal{B}_n(f)$ when $x$ is far from
the singularity of $f$. It would be interesting to establish a
precise estimate on the rate of pointwise convergence of
$\mathcal{P}_n(f)$ to explain this observation.

\item Gegenbauer and Jacobi projections are widely used in
spectral methods for differential and integral equations and their
optimal error estimates are often required in these applications.
Our work can be extended to these two cases (see
\cite{wang2020gegenbauer} for the case of Gegenbauer projections).
Following the same line of Theorem \ref{thm:LegPieceRate}, it is
possible to establish an optimal error estimate of Jacobi
projections for piecewise analytic functions by combining the result
\cite[Theorem~3]{saff1989poly} and some sharp estimates of the
Dirichlet kernel of Jacobi polynomials. Moreover, for functions of
fractional smoothness, it is also possible to establish some optimal
error estimates of Jacobi projections by combining the observation
in Remark \ref{rk:Subint} and sharp estimates of Jacobi expansion
coefficients (see \cite{xiang2018jacobi}).

\item Spectral interpolation, i.e., polynomial interpolation in roots or extrema of
Legendre, and, more generally, Gegenbauer and Jacobi polynomials, is
a powerful approach for approximating smooth functions that are
difficult to compute and serves as theoretical basis for spectral
collocation methods (see, e.g., \cite[Chapter 3]{shen2011spectral}).
It is of interest to study the comparison of the convergence
behavior of spectral interpolation and that of the best
approximation $\mathcal{B}_n(f)$ for analytic and differentiable
functions.
\end{itemize}

%

\begin{acknowledgements}
The author would like to thank two anonymous referees for their
careful reading of the manuscript and helpful comments which have
improved this paper.
\end{acknowledgements}

%
%



\end{document}